\documentclass[12pt]{amsart}
\usepackage{amssymb,latexsym}
\usepackage{enumerate}

\newtheorem{thm}{Theorem}[section]

\newtheorem{lem}[thm]{Lemma}

\newtheorem{mainthm}[thm]{Main Theorem}

\theoremstyle{definition}
\newtheorem{defin}[thm]{Definition}

\newtheorem{exa}[thm]{Example}

\newtheorem*{xrem}{Remark}

\numberwithin{equation}{section}

\begin{document}
\title[An Infinitude of Primes of the Form $b^2  + 1$]{An Infinitude of Primes of the Form $b^2  + 1$}

\author[K. Slinker]{Kent Slinker}
\address{Pima College\\RH 150\\1255 North Stone Ave.\\
Tucson, AZ 85709-3000\\USA}
\email{kslinker@pima.edu, http://kslinker.com/bsquared+1.html}

\date{August 16, 2009, revised July 26, 2010}

\begin{abstract}
If $b^2  + 1$ is prime then \textit{b} must be even, hence we examine the form $4u^2  + 1$. Rather than study primes of this form we study composites where the main theorem of this paper establishes that if $4u^2  + 1$ is composite, then \textit{u} belongs to a set whose elements are those \textit{u} such that $u^2  + t^2  = n(n + 1)$, where \textit{t} has a upper bound determined by the value of \textit{u}. This connects the composites of the form $4u^2  + 1$ with numbers expressible as the sum of two squares equal to the product of two consecutive integers. A result obtained by Gauss concerning the average number of representations of a number as the sum of two squares is then used to show that an infinite sequence of \textit{u} for which $u^2  + t^2  = n(n + 1)$ is impossible. This entails the impossibility of an infinite sequence of composites, and hence an infinitude of primes of the form $b^2  + 1$.\\
\end{abstract}
\subjclass[2000]{Primary 11A41; Secondary 11D45}

\keywords{Sums of squares, products of consecutive integers, twice triangular numbers,  Gauss circle problem, infinitude of primes, average order of sums of two squares}

\maketitle

\section{Introduction}
Are there an infinitude of primes of the form $b^2  + 1$? Hardy and Wright pose this question in Chapter 2 of their book, \textit{Introduction to the Theory of Numbers}. This question is also listed as the first problem found in Richard Guy's \textit{Unsolved Problems of Number Theory}, and is mentioned in numerous books on elementary number theory.  

Primes of the form $4m + 1$ make up one of two classes of all primes, (the other class being those of the form $4m + 3$). The fascination associated with the distribution of primes extends naturally to the distribution of \textit{m} which make up the set of those $4m + 1$ which are prime. The following is a list of both \textit{m} and $4m + 1$ for primes up to 401:\\

(\textbf{1},5), (3,13), (\textbf{4},17), (7,29), (\textbf{9},37), (10,41), (13,53), (15,61), (18,73), (22,89), (24,97), (\textbf{25},101), (27,109), (28,113), (34,137), (37,149), (39,157), (43,173), (45,181), (48,193), (\textbf{49},197), (57,229), (58,233), (60,241), (\textbf{64},257), (67,269), (69,277), (70,281), (73,293), (78,313), (79,317), (84,337), (87,349), (88,353), (93,373), (97,389), (99,397), (\textbf{100},401)\\

In boldface type are those \textit{m} which are also perfect squares.  This observation leads to the question posed by Hardy and Wright \textit{et al} concerning the infinitude of primes of the form $b^2  + 1$. Since all primes other than 2 are odd, if $b^2  + 1$ is to be prime, then \textit{b} must be even. Letting $b = 2u$ gives $4u^2  + 1$ which is the form we will examine.

We will begin by studying \textit{composites} of the form $4u^2  + 1$.  The first part of this paper will consist of proving a theorem that states for all composite numbers of the form $4u^2  + 1$, there is a \textit{t} such that $u^2  + t^2 $ is the product of consecutive integers, and \textit{t} is less than or equal to  $\frac{{u^2  - 6}}{5}$. 

For example, 65 is a composite of the form $4u^2  + 1$ where $u = 4$.  We now pose the question: is there a number \textit{t}  such that $4^2  + t^2  = n(n + 1)$ for some \textit{n} and, if yes, is $\textit{t}\leq\frac{4^2-6}{5}$? The answer is yes, $4^2+2^2=20=4(5)$ and $2 \le {{4^2  - 6} \over 5} = 2$

Another example:  145 is a composite of the form $4u^2  + 1$ where $u = 6$.  The question is: is there a number \textit{t}  such that $6^2  + t^2  = n(n + 1)$ for some \textit{n} and is $\textit{t} \le {{4^2  - 6} \over 5}$?  The answer is yes, $6^2  + 6^2  = 72 = 8(9)$ and $6 \le {{6^2  - 6} \over 5} = 6$.  

Let us form a set of all numbers \textit{u}, such that if there is a \textit{t} where 
$u^2  + t^2$ is the product of consecutive integers, and \textit{t} is less than or equal to the integral part of  ${{u^2  - 6} \over 5}$, then \textit{u} is a member of our set. Let us call this set \textbf{S}.  

We can formally define this set as follows:

\begin{defin}
Where  $\{ u,t,n\}  \in {\rm{\textbf{N}}}$ with ${\rm{ \textit{t}}} \le{{{u^2  - 6} \over 5}}$ then:
\begin{equation}
\textbf{S} = \left\{u:u^2  + t^2  = n(n + 1) \right\}
\end{equation}
\end{defin}

\begin{exa}
Consider $11^2  + 23^2  = 25(26)$, since \textit{u} is an element of \textbf{S} provided that $t \le {{{u^2  - 6} \over 5}}$, 
and $\frac{{11^2  - 6}}{5} = 23$
, 11 is an element of \textbf{S}.
\end{exa}

\begin{exa}
The law of commutativity in Example 1.2 allows us to exchange the value of \textit{u} with \textit{t}, to obtain $23^2  + 11^2  = 25(26)$.  Again, \textit{u} is an element of \textbf{S} provided that $\textit{t} \le {{{u^2  - 6} \over 5}}$, and ${{{23^2  - 6} \over 5}} \approx 104$, hence 23 is also an element of \textbf{S}.
\end{exa}

\begin{exa}
Now consider $5^2  + 25^2  = 25(26)$. However, in this case ${{{5^2  - 6} \over 5}} = 3$, since 25 is not less than or equal to 3,  then $5 \notin \textbf{S}$. 
However, if we use the law of commutativity and let $u = 25$ and $t = 5$, we have:
$25^2  + 5^2  = 25(26)$  and in this case $
{{{25^2  - 6} \over 5}} \approx 123
$, and $
{\rm{5}} \le {\rm{123}}
$, so $25 \in \textbf{S}$.
\end{exa}
From the above three examples, we can conclude that: $4\left( {11} \right)^2  + 1$, $4\left( {23} \right)^2  + 1$, and $4\left( {25} \right)^2  + 1$ are all composite. Indeed, $4\left( {11} \right)^2  + 1 = (5)(97)$, $4\left( {23} \right)^2  + 1 = (29)(73)$, and $4\left( {25} \right)^2  + 1 = (41)(61)$.

Our first step will be to prove that for all composite $4u^2  + 1$, $u \in \textbf{S}$.
We will do so by assuming $4u^2  + 1$ is composite, and prove the individual requirements for set membership in \textbf{S} given by Definition 1.1. 
							
We consider two exhaustive cases: $4u^2  + 1$ is the product of two factors (not necessarily prime) congruent to 1 (mod 4) or congruent to 3 (mod 4). \\
\textbf{
\begin{center}
Case I: \textbf{$4u^2  + 1$} is the product of 2 factors congruent to 1 (mod 4)
\end{center}
}

Assuming $4u^2  + 1$  is composite we have: 
\begin{defin}
$4u^2  + 1 = (4m + 1)(4k + 1)$ 
\end{defin}
We now note that $m \ne k$, since if they were equal, then we would have a perfect square, which is impossible, as  $4u^2  + 1$   is one greater than a perfect square. Since  $m \ne k$, one must be greater than the other, so we will stipulate that $m < k$		

Since $m < k$, we designate its difference as \textit{t}.
\begin{defin}
$k = t + m$
\end{defin}
We now substitute $k = t + m$ into $4u^2  + 1 = (4m + 1)(4k + 1)$ and solve for \textit{m} using the quadratic formula to obtain: 
\begin{equation}
m = \frac{1}{4}\left( {\sqrt {1 + 4(u^2  + t^2 )}  - (2t + 1)} \right)
\end{equation}
Since \textit{m} is an integer, and $2t + 1$ is odd, then $\sqrt {1 + 4(u^2  + t^2 )} $ \textit{must be odd, and an integer whose difference with} $2t + 1$  \textit{is congruent to 0 (mod 4)}.  We will prove the mod 4 congruency after Lemma 1.9\\

Since $\sqrt {1 + 4(u^2  + t^2 )}$ must be odd, let  $2n + 1 = \sqrt {1 + 4(u^2  + t^2 )} $.
Squaring both sides gives $4n^2  + 4n + 1 = 1 + 4(u^2  + t^2 )$ which simplifies to:

\begin{equation}
u^2  + t^2  = n(n + 1)
\end{equation}
Equation 1.3 establishes the main condition for set membership in \textbf{S}.						
In order to prove the mod 4 congruency mentioned above, we must first establish some lemmas concerning the possible values of \textit{n} in Equation 1.3

\begin{lem}
u and t have the same parity 
\end{lem}

\begin{proof}
Examining Equation 1.3, $u^2  + t^2  = n(n + 1)$, we notice that the right hand side is always even because either \textit{n} or $n + 1$ is even, hence both \textit{u} and \textit{t} must have the same parity.
\end{proof}

\begin{lem}
Neither \textit{n} or $(n + 1)$ can be congruent to 3 (mod 4)				
\end{lem}

\begin{proof}
Fermat proved that a number can be expressed as the sum of two squares only if its primes congruent to 3 (mod 4) are raised to even powers.
Since \textit{n} and $(n + 1)$  are relatively prime, they share no common factors (any factors which divide one will have a remainder of either 1 or -1  when they divide the other).  
If either \textit{n} or $(n + 1)$  is congruent to 3 (mod 4), then it has at least one factor congruent to 3 (mod 4) which is raised to an odd power. The only way that an odd powered factor could become even and thus make $n(n + 1)$ expressible as the sum of two squares is if it gains an odd prime power from \textit{n} or $(n + 1)$  . But this is impossible, since \textit{n} and $(n + 1)$   are relatively prime. Hence neither \textit{n} or $(n + 1)$   can be congruent to 3 (mod 4).
\end{proof}

\begin{lem}
If \textit{u} and \textit{t} are odd, then \textit{n} must be odd and congruent to 1 (mod 4) 
\end{lem}

\begin{proof}
We expand the right hand side of  Equation 1.3 and analyze mod 4:\\

$u^2  + t^2  \equiv n^2  + n{\text{   (mod 4)}}$\\

If  \textit{u} and \textit{t} are odd, we have:\\

$2 \equiv n^2  + n{\text{   (mod 4)}}$ \\

Rearranging terms we have:\\

$2 - n \equiv n^2 {\text{  (mod 4)}}$\\

If \textit{n} is even, then we have: $2 - n \equiv 0{\text{  (mod 4)}}$, which means \textit{n} must be congruent to 2 (mod 4) and $(n + 1) \equiv 3{\text{ (mod 4)}}$, which is impossible by Lemma 1.8\\

Letting \textit{n} be odd we have: \\

$2 - n \equiv 1{\text{  (mod 4)}}$ \\

Hence \textit{n} must be congruent to 1 (mod 4) if \textit{u} and \textit{t} are odd.
\end{proof}

It follows immediately from Lemmas 1.8 and 1.9  that:

\begin{lem}
If u and t are even, then n must be even and congruent to 0 (mod 4) 			
\end{lem}

We can now prove that $2n + 1 - \left( {2t + 1} \right) \equiv 0{\text{ (mod 4)}}$

\begin{proof}
Let \textit{t} be odd: $t = 2w + 1$. From Lemma 1.9, \textit{n} must be odd and congruent to 1 (mod 4).  Let $n = 4v + 1$, so:\\

$2n + 1 - \left( {2t + 1} \right) = 2(n - t)
$, but $t = 2w + 1$ and $n = 4v + 1$, so:\\

$2(n - t) = 2(4v + 1 + 2w + 1)$\\

$2(4v + 1 + 2w + 1) = 4(2v + w + 1)$\\

$4(2v + w + 1) \equiv 0{\text{ (mod 4)}}$\\

Now let \textit{t} be even: $t = 2w$. From Lemma 1.10, \textit{n} must be even and congruent to 0 (mod 4). Hence, $n = 4v$, so:\\

$2n + 1 - \left( {2t + 1} \right) = 2(n - t)$, but $t = 2w$ and $n = 4v$, so:\\

$2(n - t) = 2(4v + 2w)$\\

$2(4v + 2w) = 4(2v + w)$\\

$4(2v + w) \equiv 0{\text{ (mod 4)}}$\\
\end{proof}

We now turn to Case II, where $4u^2  + 1$ is the product of two factors congruent to 3 (mod 4)\\

\begin{center}
\textbf{Case II: \textbf{$4u^2  + 1$} is the product of 2 factors congruent to 3 (mod 4)}
\end{center}

We proceed exactly as in Case I above. However, the end result will be that Case II is impossible. 
\begin{defin}
$4u^2  + 1 = (4m + 3)(4k + 3)$
\end{defin}
Again, $m \ne k$, and we stipulate $m < k$.
We let $k = t + m$, and solve $4u^2  + 1 = (4m + 3)(4k + 3)$ for \textit{m} after replacing $k = t + m$. This gives:

\begin{equation}
m = \frac{1}
{4}\left( {\sqrt {1 + 4(u^2  + t^2 )}  - (2t + 3)} \right)
\end{equation}

Again, since \textit{m} is an integer, and $2t + 3$ is odd, then $\sqrt {1 + 4(u^2  + t^2 )} $ \textit{must be odd, and an integer whose difference with} $2t + 3$ \textit{is congruent to 0 (mod 4)}.\\

Let  $2n + 1 = \sqrt {1 + 4(u^2  + t^2 )}$. Squaring both sides gives $4n^2  + 4n + 1 = 1 + 4(u^2  + t^2 )$, which simplifies to $u^2  + t^2  = n(n + 1)$, the same as Case I.\\

We can now prove that $2n + 1 - (2t + 3) \not\equiv 0{\text{ (mod 4)}}$

\begin{proof}
Let \textit{t} be odd: $t = 2w + 1$. From Lemma 1.9, \textit{n} must be odd and congruent to 1 (mod 4).  Let $n = 4v + 1$, so:\\

$2n + 1 - (2t + 3) = 2(n - t - 1)$, but $t = 2w + 1$ and $n = 4v + 1$, so:\\

$2(n - t - 1) = 2(4v + 1 - (2w + 1) - 1)$\\

$2(4v + 1 - (2w + 1) - 1) = 2(4v + 1 - 2w - 1 - 1)$\\

$2(4v + 1 - 2w - 1 - 1) \equiv 2(4v - 2w - 1)$\\

$ 2(4v - 2w - 1) = 8v - 4w - 2$ \\

$8v - 4w - 2 \equiv  - 2{\text{ (mod 4)}}$\\

Hence $2n + 1 - (2t + 3) \not\equiv 0{\text{ (mod 4)}}$ with \textit{t} odd. \\

Now let \textit{t} be even: $t = 2w$. From Lemma 1.10, \textit{n} must be even and congruent to 0 (mod 4). Hence, $n = 4v$, so:\\

$2n + 1 - (2t + 3) = 2(n - t - 1)$, but $t = 2w$ and $n = 4v$, so:\\

$2(n - t - 1) = 2(4v - 2w - 1)$ so:\\

$2(4v - 2w - 1) = 8v - 4w - 2$\\
 
$8v - 4w - 2 \equiv  - 2{\text{ (mod 4)}}$\\

Hence $2n + 1 - (2t + 3) \not\equiv 0{\text{ (mod 4)}}$ with \textit{t} even.\\

Since \textit{t} must be even or odd, and in neither case is $2n + 1 - \left( {2t + 3} \right) \equiv 0{\text{ (mod 4)}}$, then Case II is impossible, which means that $4u^2  + 1$ cannot be the product of two factors congruent to 3 (mod 4).

\end{proof}

\textbf{The value of \textbf{$t_{\max} $}}

We now establish the upper limit for \textit{t} given in the definition of set \textbf{S}.
We note that:\\

$4u^2  + 1 = (4m + 1)(4k + 1) \to u^2  = k(4m + 1) + m$\\

First we establish the maximum value of \textit{k}, which we designate as $k_{\max } $.
Since 5 is the first prime congruent to 1 (mod 4), the least possible value of \textit{m} is 1.  Since the maximum value of $4k + 1$ is obtained when the value $4m + 1$ is at its minimum, the maximum value ($k_{\max } $) is given by setting $m = 1$, this gives $u^2  = k_{\max } (5) + 1$, rearranging terms we have:

\begin{equation}
k_{\max }  = \frac{{u^2  - 1}}{5}			
\end{equation}

Since $t = k - m$, we obtain by substitution of $k_{\max } $ and $m = 1$:

\begin{equation}
t_{\max }  = \frac{{u^2  - 6}}{5}
\end{equation}\\

This establishes the upper limit for \textit{t} found in the universe of discourse for the set \textbf{S}.\\ 

\textbf{The value of $t_{\max } $ prevents the trivial case where $t = u^2$}\\

Consider the case where $t = u^2 $. This allows us to write:\\

$u^2  + t^2  \to u^2  + (u^2 )^2  = u^2 (u^2  + 1)$\\

Substitution of $u^2 = n$ gives the tautology $n(n+1) = n(n + 1)$. If this case were allowed, then every \textit{u} would be an element of \textbf{S} by selecting a value of $t = u^2 $. However, the value of $t_{\max } $ prevents this since $
u^2  >  {{{u^2  - 6} \over 5}} = t_{\max } 
$\\

We now establish the main theorem of this paper, which connects all composite instances of $4u^2  + 1$ with the sum of two squares equal to the product of two consecutive integers.\\

\begin{mainthm}
$4u^2  + 1$ is composite if and only if $u \in \textbf{S}$
\end{mainthm}

\begin{proof}
Suppose to the contrary that $u \notin \textbf{S}$, but $4u^2  + 1$ is composite. However, if $4u^2  + 1$ is composite then by stipulation one factor is given as $4m + 1$, hence \textit{m} must exist, Equation 1.6 sets the limit for $t_{\max} $ provided \textit{m} exists and equations 1.2 and 1.3 entail that $u \in \textbf{S}$, which contradicts our original assumption that $u \notin \textbf{S}$.

\end{proof}

\begin{xrem}
It might be correctly stated that the open question concerning the finitude of primes of the form $4u^2  + 1$ is due to the fact that we have a squared term. Our above analysis takes advantage of this by assuming $4u^2  + 1 = (4m + 1)(4k + 1)$ and noting that $m \ne k$. This leads to the fruitful result obtained by letting $k = t + m$ and results in the elements which define the set \textbf{S}.  Additionally the move from $b^2  + 1$ to $4u^2  + 1$ sets up a one-to-one mapping of \textit{u} to the number line via the elements of \textbf{S}, which sets the stage for the analysis which follows below. 
\end{xrem}

\section{The average order of the sum of two squares and the infinitude of primes of the form $4u^2  + 1$.}

\begin{exa}
We explore the implications of Theorem 1.12 by an example.
Consider the following  products of consecutive integers: $ (40)(41) $, $ (52)(53) $, $ (232)(233) $.

Each can be represented as a sum of two squares in the first octant in two distinct ways, namely:\\

$(40)(41) = 14^2  + 48^2 $, $22^2  + 34^2 $

$(52)(53) = 16^2  + 50^2 $, $30^2  + 34^2 $

$(232)(233) = 34^2  + 230^2 $, $134^2  + 190^2 $
\\

Hence, by Theorem 1.12 (plus commutativity and the fact the the maximum value of $t$ is not exceeded in each case, review Examples 1.3 and 1.4 above), when $u =${$14$, $48$, $22$, $34$, $16$, $50$, $30$, $230$, $134$, and $190$}, then $4u^2  + 1$ is composite for each $u$ in that list.  Observe the element $u = 34 $ appears  as one of the two squares in each representation of $ (40)(41) $, $ (52)(53) $, and $ (232)(233) $, hence of the total \textit{12 possible additions} to the set \textbf{S} from the \textit{6 different} representations of $n(n+1) $ as a sum of two squares \textit{only 10} elements are added to the set \textbf{S}.

\end{exa}

Since the set \textbf{S} contains all composite solutions to $4u^2  + 1$, the \textit{finitude} of primes of the form $4u^2  + 1$  entails that there must be an infinite series of successive integers which are all elements of \textbf{S}. In the language of set theory, this means that there must be a subset of \textbf{S}, which we will designate as \textbf{L}, with a least member which we will designate as $x_0 $,  such that every integer $ \geq  x_0 $ is contained in \textbf{L}. Stated another way, for \textit{every integer} $x \geq  x_0 $, there is a $u=x$ such that $u^2  + t^2  = n(n + 1)$.

We now formally define the set \textbf{L}:\\

Where $\textbf{L} \subset \textbf{S}$ and $x_0$  is the least member of \textbf{L} and $x \in \textbf{N}$, then:

\begin{defin}
$\textbf{L} = \left\{ {x:x \geq  x_0 } \right\}$
\end{defin}

We now prove that \textbf{L} cannot exist (we follow Grosswald [2] closely for what follows with respect to the average order of the sum of two squares, see also Hardy [5]). Our principle result will show that the main theorem of this paper, Gauss's proof concerning the average number of ways $m$ can be written as the sum of two squares as $m \to \infty $, and an infinite sequence of composites of the form $4u^2  + 1$ are inconsistent.\\

The average order for the sum of two squares was shown by Gauss to be $\pi  + \varepsilon $, where the determination of $\varepsilon  \ll 1$ constitutes what is known as the Gauss circle problem.  
As is customary we will denote $r_2 (m)$ as the number of ways \textit{m} can be expressed as the sum of 2 squares, and let $\sum\limits_{m < d} {r_2 } (m) = A(d)$ be the total number of such representations for all numbers $ < d$. The customary treatment of $r_2 (m)$ ignores the law of commutativity and signs such that:

\[\overbrace {{\text{quadrant 1}}}^{\left\{ {{\text{(x, y), (y, x)}}} \right\}}{\text{  }}\overbrace {{\text{quadrant 2}}}^{\left\{ {{\text{( - x, y), ( - y, x)}}} \right\}}{\text{  }}\overbrace {{\text{quadrant 3}}}^{\left\{ {{\text{( - x,  - y), ( - y,  - x)}}} \right\}}{\text{  }}\overbrace {{\text{quadrant 4}}}^{\left\{ {{\text{(x,  - y), (y,  - x)}}} \right\}}{\text{ }}\]
\\

are counted as 8 distinct representations.  Hence $\frac{{A(d)}}
{d}$ is the average number of all such representations up to \textit{d}, $\frac{{A(d)}}
{{4d}}$ is the average number of such representations in the first quadrant, and $\frac{{A(d)}}
{{8d}}$ is the average number of \textit{individual representations}. 
Now $\frac{{A(d)}}
{d} = \pi  + \varepsilon $, so $\frac{{A(d)}}
{{8d}} = \frac{{\pi  + \varepsilon }}
{8}$. Since $\varepsilon  \ll 1$ it can be ignored, since it must be greater than $4 - \pi $ to affect our proof, which is not the case.\footnote{In other words, ${1 \over 2} - {{\pi  + \varepsilon } \over 8} = 0$ gives $4 - \pi  = \varepsilon $. Hardy and Wright give $
\varepsilon  = O\left( {{{\sqrt d } \over d}} \right)
$
, which has been greatly improved since.}  Hence we just write $\frac{{A(d)}}
{8} = \frac{\pi }
{8}$.\\

We now use this result to prove the impossibility of \textbf{L}.
\begin{proof}
Recall Definition 2.1, $\textbf{L} = \left\{ {x:x \geq x_0 } \right\}$ where $\textbf{L} \subset \textbf{S}$. 
Let us assume the \textit{minimal} requirement to establish an infinite sequence of composites of the form $4u^2  + 1$. For every $u^2  + t^2  = n(n + 1)$, both \textit{u} and \textit{t} would be, under this assumption, \textit{unique} elements of \textbf{L}. In other words the minimal case assumes there are no  instances like $19^2  + 17^2  = 25(26)$ and $19^2  + 71^2  = 73(74)$ which contribute only \textit{three} members to \textbf{S} (recall Example 2.1 above).
The minimal existence of \textbf{L}, in turn requires that the ratio between individual representations of $n(n+1)$ as the sum of two squares and the natural numbers greater than $x_0 $ be \textit{at the bare minimum} 1 to 2. But this ratio is impossible \textit {even for the set of all integers} expressible as the sum of two squares, much less those equal to $n(n+1)$, which are indeed a sparse subset of all integers! The average number of representations for all integers expressible as the sum of two squares in the first octant as we have seen is ${\pi  \over 8}$ which is less than the required ratio of ${1 \over 2}$ for \textbf{L}, which itself is a subset of all numbers expressible as the sum of two squares. Since the minimal case represents the least ratio, any other case requires a greater ratio than ${1 \over 2}$, the subset \textbf{L} of \textbf{S} is impossible.
\end{proof}

By assuming the existence of the set \textbf{L} we have assumed the \textit{finitude} of primes of the form $4u^2  + 1$. But the set \textbf{ L} has been shown to be impossible, hence there must be an infinitude of primes of the form $4u^2  + 1$.

\subsection*{Acknowledgements}
My sincere thanks to Frederick W. Stevenson of the University of Arizona Mathematics Department for his many helpful suggestions, and Aeyn Edwards-Wheat for his never ending encouragement and his many editorial suggestions.

\end{document}